\documentclass[11pt]{article}
\usepackage{graphicx}
\usepackage{amsmath,amsfonts,amssymb,amsthm}
\usepackage{bbm,mathrsfs}
\usepackage{bm}
\title{Intersection probabilities and kinematic formulas\\for polyhedral cones}
\author{Rolf Schneider}
\date{}
\newcommand{\Sd}{{\mathbb S}^{d-1}}
\newcommand{\R}{{\mathbb R}}

\newcommand{\C}{{\mathcal C}}

\newcommand{\bP}{{\mathbb P}}

\newcommand{\Rd}{{\mathbb R}^d}
\newcommand{\N}{{\mathbb N}}

\newcommand{\Ha}{\mathcal{H}}

\newcommand{\B}{\mathcal{B}}
\newcommand{\D}{{\rm d}}
\newcommand{\F}{{\mathcal F}}

\newcommand{\PC}{{\mathcal PC}}

\newcommand{\bE}{{\mathbb E}\,}

\newcommand{\bQ}{{\mathbb Q}}
\newcommand{\g}{{\bf g}}

  \renewcommand{\dim}{{\rm dim}\,}
  
  \newcommand{\relint}{{\rm relint}\,}

  \newcommand{\lin}{{\rm lin}\,}

\newtheorem{theorem}{Theorem}[section]
\newtheorem{lemma}{Lemma}[section]

\begin{document}
\maketitle

\begin{abstract}
For polyhedral convex cones in $\R^d$, we give a proof for the conic kinematic formula for conic curvature measures, which avoids the use of characterization theorems. For the random cones defined as typical cones of an isotropic random central hyperplane arrangement, we find probabilities for non-trivial intersection, either with a fixed cone, or for two independent random cones of this type.\\[1mm]
\end{abstract}

\section{Introduction}\label{sec1}

The use of convex optimization for signal demixing under a certain random model has brought into focus the following geometric question, as formulated in \cite[p. 518]{MT14a}: ``When does a randomly oriented cone strike a fixed cone?'' More precisely, let $C,D\subset \R^d$ be closed convex cones, not both subspaces. Let $\bm \theta$ be a uniform random rotation, that is, a random element of the rotation group ${\rm SO}_d$ of $\R^d$ with distribution equal to the normalized Haar measure $\nu$ on ${\rm SO}_d$. The question asks for the probability
\begin{equation}\label{1.1} 
\bP\{C\cap {\bm \theta} D\not=\{o\}\}.
\end{equation}
Various aspects of this question have recently received considerable attention in connection with conic optimization; see, for example, \cite{AB12, AB15, ALMT14, MT14a, MT13}.

The natural approach to the evaluation of the probability (\ref{1.1}) uses spherical, or conic, integral geometry. For this, one needs the conic intrinsic volumes $V_1(C),\dots,V_d(C)$ of a closed convex cone $C\subset\R^d$ (see Section \ref{sec3}). The conic kinematic formula provides the expectation
\begin{equation}\label{1.2}
\bE V_k(C\cap {\bm \theta} D)= \sum_{i=k}^d V_i(C)V_{d+k-i}(D)
\end{equation}
for $k=1,\dots,d$. A version of the spherical Gauss--Bonnet theorem says that
\begin{equation}\label{1.3} 
2\sum_{k=0}^{\lfloor\frac{d-1}{2}\rfloor} V_{2k+1}(C)=1,\quad\mbox{if $C$ is not a subspace}.
\end{equation}
(See, e.g., \cite[Thm. 6.5.5]{SW08}, for an equivalent formulation in the spherical setting. Differential-geometric versions of the formula appear in early work of Santal\'{o}, e.g. \cite{San62a}; for a combinatorial approach, see McMullen \cite{McM75}.) 
Since $C\cap{\bm \theta} D$ is, with probability one, either $\{o\}$ (in which case $V_k(C\cap{\bm \theta} D)=0$ for $k\ge 1$) or not a subspace, (\ref{1.3}) implies that
$$ {\mathbbm 1}\{C\cap {\bm \theta} D\not=\{o\}\} = 2\sum_{k=0}^{\lfloor\frac{d-1}{2}\rfloor} V_{2k+1}(C\cap {\bm \theta} D)\quad\mbox{almost surely}.$$
Therefore, (\ref{1.2}) yields
\begin{equation}\label{1.4} 
\bP\{C\cap{\bm \theta} D\not=\{o\}\} = 2\sum_{k= 0}^{\lfloor\frac{d-1}{2}\rfloor} \sum_{i=2k+1}^d V_i(C)V_{d+2k+1-i}(D).
\end{equation}
A major concern of the quoted investigations is the computation of the conic intrinsic volumes for some special cones, or their estimation in the case of general cones.

In the present paper, we restrict ourselves to polyhedral cones and are interested in variants of (\ref{1.2}) and (\ref{1.4}), which are of interest from the viewpoint of stochastic and integral geometry. First, the conic kinematic formula (\ref{1.2}) has a local version. It involves the conic curvature measures $\Phi_1(C,\cdot),\dots,\Phi_d(C,\cdot)$ of a polyhedral cone $C$. These are finite measures on the $\sigma$-algebra $\widehat\B(\R^d)$ of conic Borel sets in $\R^d$ (see Section \ref{sec3}). Since $\Phi_k(C,\R^d)=V_k(C)$, the global case (i.e., $A=B=\R^d$) of the following result yields (\ref{1.2}).

\begin{theorem}\label{T1.1}
Let $C,D\subset\R^d$ be polyhedral cones, and let $A,B\in\widehat\B(\R^d)$ be conic Borel sets. Then
\begin{equation}\label{4.5} 
\int_{{\rm SO}_d} \Phi_k(C\cap\vartheta D,A\cap\vartheta B)\,\nu(\D\vartheta) =\sum_{i=k}^d \Phi_i(C,A)\Phi_{d+k-i}(D,B)
\end{equation}
for $k=1,\dots,d$.
\end{theorem}

Theorem \ref{T1.1} has an equivalent formulation within spherical integral geometry, for  spherically convex polytopes in the unit sphere $\Sd$ of $\R^d$ and their spherical curvature measures. In this form, the theorem was proved, in two different ways, by Glasauer \cite{Gla95}. (A summary of this thesis appears in \cite{Gla96}. For smooth submanifolds, spherical integral geometry goes back to Santal\'{o}; see, e.g., \cite{San62b, San76}.) Glasauer's first proof \cite[Thm. 5.1.1]{Gla95} is based on an axiomatic characterization of the spherical curvature measures (similar to the treatment of the Euclidean case in \cite[Sect. 7]{Sch78}). This proof is reproduced in \cite[Thm. 6.5.6]{SW08}. Glasauer's second proof \cite[Thm. 6.1.1]{Gla95} was obtained by specializing his kinematic formula for spherical support measures. For the proof of the latter, Glasauer used his new axiomatic characterization of these support measures. This second proof was transferred to the conic situation, and the result expanded, by Amelunxen \cite{Ame15}. Also the local conic kinematic formula has proved relevant for applications, see \cite{AB15}.

The approach to the mentioned proofs of the local kinematic formula, via a characterization theorem for curvature measures or support measures, may be elegant (and has important  examples in Hadwiger's work), but it has its limits. For instance, it seems to be inefficient for the extension of kinematic formulas to the tensor-valued generalizations of the intrinsic volumes. For these, the existing proofs in the Euclidean case are still rather complicated, see \cite{HSS08} and \cite{HW16}. For that reason, approaches to local kinematic formulas by direct computation should be given more attention. For the global conic kinematic formula, an elegant proof of this kind has recently been found by Amelunxen and Lotz \cite{AL15}. We extend their approach here to a proof of Theorem \ref{T1.1}. There is some hope that this approach will lead to simplifications in the integral geometry of tensor valuations.

We restrict ourselves here to polyhedral convex cones. The extension to general closed convex cones, where possible, can be done along the lines of Glasauer's \cite{Gla95} argumentation, properly transferred to the conic case.

Our second aim in this paper is a variant of (\ref{1.4}). This formula gives the probability of non-trivial intersection of a random cone with a fixed cone. The random cone is of a special type: the randomness comes only from the random rotation, which is applied to a fixed cone. It is certainly of interest to have similar results for more flexible types of random cones, where also the shape can be random and not only the position. On the other hand, it is to be expected that only very special models can lead to explicit results. We observe here that such explicit results can be obtained for the isotropic random Schl\"afli cones, which were studied in \cite{HS16}. 

To recall their definition, let $H_1,\dots,H_n$ be hyperplanes in $\R^d$ through the origin $o$ which are in general position, that is, any $k\le d$ of them have an intersection of dimension $d-k$. The {\em Schl\"afli cones} induced by $H_1,\dots,H_n$ are the closures of the components of $\R^d\setminus\bigcup_{i=1}^n H_i$. The name has been chosen since Schl\"afli has proved that there are exactly
\begin{equation}\label{1.0} 
C(n,d):= 2\sum_{r=0}^{d-1} \binom{n-1}{r}
\end{equation}
of them. If we take $n$ stochastically independent, isotropic random hyperplanes through $o$ and choose at random, with equal chances, one of the Schl\"afli cones induced by them, then this defines the isotropic random Schl\"afli cone $S_n$, with parameter $n$.

We show the following counterpart to Theorem \ref{T1.1}, where the randomly rotated cone is now replaced by an isotropic random Schl\"afli cone.

\begin{theorem}\label{T1.2}
Let $C\subset\R^d$ be a polyhedral cone and let $A\in\widehat\B(\R^d)$ be a conic Borel set. Let $n\in \N$ and $k\in\{1,\dots,d\}$. Then
$$ \bE\Phi_k(C\cap S_n,A) = \frac{1}{C(n,d)} \sum_{s=0}^{\min\{n,d-k\}} \binom{n}{s} \Phi_{k+s}(C,A).$$
\end{theorem}

The following intersection probabilities are obtained.

\begin{theorem}\label{T1.3}
Let $C\subset \R^d$ be a closed convex cone, not a subspace. Then the isotropic random Schl\"afli cone $S_n$ satisfies
$$ \bP\{C\cap S_n\not=\{o\}\}=\frac{2}{C(n,d)} \sum_{j=1}^n \sum_{k=0}^{\lfloor \frac{j-1}{2}\rfloor} \binom{n}{j-2k-1}V_j(C).$$
\end{theorem}

\begin{theorem}\label{T1.4}
If $S_n, T_m$ are stochastically independent isotropic random Schl\"afli cones with parameters $n,m$, respectively, then
$$\bP\{ S_n\cap T_m\not=\{o\}\} = \frac{2}{C(n,d)C(m,d)}  \sum_{k=0}^{\lfloor \frac{d-1}{2}\rfloor} \sum_{p+q=d-2k-1} \binom{n}{p}\binom{m}{q}.$$
\end{theorem}

After some preliminaries in the next section, we recall the conic intrinsic volumes and curvature measures in Section \ref{sec3}. Theorem \ref{T1.1} will then be proved in Section \ref{sec4}, and the proofs of Theorems \ref{T1.2}, \ref{T1.3}, \ref{T1.4} follow in Section \ref{sec5}. 

\section{Preliminaries}\label{sec2}

We work in the $d$-dimensional real vector space $\R^d$ ($d\ge 2$), with scalar product $\langle\cdot\,,\cdot\rangle$ and induced norm $\|\cdot\|$.  Lebesgue measure on a $k$-dimensional subspace, which will be clear from the context, is denoted by $\lambda_k$. Spherical Lebesgue measure on the unit sphere $\Sd:=\{x\in \R^d:\|x\|= 1\}$ is denoted by $\sigma_{d-1}$. Its total measure is given by the constant
$$ \omega_d =\sigma_{d-1}(\Sd)= \frac{2\pi^{\frac{d}{2}}}{\Gamma\left(\frac{d}{2}\right)}.$$
The standard Gaussian measure on a $k$-dimensional linear subspace $L\subseteq\R^d$ is the probability measure defined by
$$ \gamma_L(A)=\frac{1}{\sqrt{2\pi}^{\,k}}\int_A e^{-\frac{1}{2}\|x\|^2}\,\lambda_k(\D x)$$
for $A\in\B(L)$. We write $\gamma_{\R^d}=\gamma_d$. The following property of Gaussian measures is frequently used. Suppose that $A, A'$ are Borel sets which lie in totally orthogonal subspaces $L,L'$ of $\R^d$. Then
\begin{equation}\label{2.0}
\gamma_d(A+A')= \gamma_{L}(A)\gamma_{L'}(A').
\end{equation}

The orthogonal group ${\rm O}_d$ of $\R^d$ is equipped with its usual topology. By ${\rm SO}_d$ we denote its subgroup of proper rotations, and by $\nu$ the Haar probability measure on ${\rm SO}_d$. If $L\subset \R^d$ is a subspace, we denote by ${\rm SO}_{L}$ the subgroup of ${\rm SO}_d$ that fixes $L^\perp$ pointwise and hence maps $L$ into itself. It is isomorphic to the group of proper rotations of $L$; its Haar probability measure is denoted by $\nu_L$. 

For a topological space $X$, we denote by $\B(X)$ the $\sigma$-algebra of its Borel subsets.

For a Borel set $A\in\B(\Sd)$ and an arbitrary vector $u\in \Sd$, we have
\begin{equation}\label{2.1}
\int_{{\rm SO}_d} {\mathbbm 1}_A(\vartheta u)\,\nu(\D \vartheta)= \frac{\sigma_{d-1}(A)}{\omega_d},
\end{equation}
as follows easily from known uniqueness theorems for invariant measures. As usual, ${\mathbbm 1}_A$ denotes the indicator function of the set $A$. 

A subset $A\subseteq \R^d$ is called {\em conic} if $a\in A$ implies $\lambda a\in A$ for all $\lambda>0$. With every subset $A\subseteq \R^d$ we associate the conic set 
$$ A^+:=\{\lambda a: a\in A,\,\lambda>0\}.$$
For $A\subseteq\Sd$, we then have $A=A^+\cap\Sd$, and the map $A\mapsto A^+$ is a bijection between the subsets of $\Sd$ and the conic subsets of $\R^d$ not containing $o$. 

The {\em conic (Borel) $\sigma$-algebra} of $\Rd$ is defined by
$$ \widehat\B(\R^d):= \{A\in\B(\R^d): A^+=A\}.$$
Similarly, for $\eta\subseteq\R^d\times\R^d$, we write
$$ \eta^+:= \{(\lambda x,\mu y): (x,y)\in\eta,\,\lambda,\mu>0\}.$$
We equip the Cartesian product $\R^d\times\R^d$ with the product topology and define the {\em biconic (Borel) $\sigma$-algebra}
$$ \widehat \B(\R^d\times\R^d):= \{\eta\in\B(\R^d\times\R^d): \eta^+=\eta\}.$$
It is clear that $\widehat\B(\R^d)$ and $\widehat \B(\R^d\times\R^d)$ are $\sigma$-algebras. (The terms `conic $\sigma$-algebra' and `biconic $\sigma$-algebra' were suggested in \cite{AB12}.)

For $A\in\widehat\B(\R^d)$, one sees by using polar coordinates that
\begin{equation}\label{2.3}
\gamma_d(A)= \frac{\sigma_{d-1}(A\cap\Sd)}{\omega_d},
\end{equation}
and that, for arbitrary $x\in\R^d\setminus\{o\}$, we have
\begin{equation}\label{2.4}
\int_{{\rm SO}_d} {\mathbbm 1}_A(\vartheta x)\,\nu(\D\vartheta) =\gamma_d(A).
\end{equation}
This follows from (\ref{2.1}) and (\ref{2.3}), since $ {\mathbbm 1}_A(\vartheta (x/\|x\|))= {\mathbbm 1}_A(\vartheta x)$ for a conic set $A$.

The correspondences $A\leftrightarrow A^+$ and $\sigma_{d-1}(A)/\omega_d=\gamma_d(A^+)$ for $A\subseteq\Sd$ allow one to translate spherical integral geometry into conic integral geometry.

\section{Conic intrinsic volumes and curvature measures}\label{sec3}

In this section, we recall some basic facts about convex cones, conic intrinsic volumes and conic curvature measures. Proofs and references can be found in Glasauer \cite{Gla95}, Schneider and Weil \cite[Sect. 6.5]{SW08}, Amelunxen and Lotz \cite{AL15}. The spherical setting in \cite{Gla95} and \cite{SW08} is easily transferred to the conic setting. We borrow some notation and formulations from the elegant presentation \cite{AL15}.

A subset $C\subseteq\R^d$ is a convex cone if it is nonempty and is closed under vector addition and under multiplication by nonnegative real numbers. For such a cone $C$, we write $L(C)$ for its linear hull. We denote by $\C^d$ the set of closed convex cones in $\R^d$. 

By $\PC^d\subset\C^d$ we denote the subset of polyhedral cones. A cone is polyhedral if it is the intersection of finitely many closed halfspaces, or equal to $\R^d$. For $C\in\C^d$, the dual or polar cone is defined by
$$ C^\circ = \{x\in\R^d: \langle x,y\rangle\le 0\mbox{ for all }y\in C\}.$$
It satisfies $C^\circ\in\C^d$ and $C^{\circ\circ}:=(C^\circ)^\circ=C$.

The nearest-point map (or metric projection) of $C\in\C^d$ is denoted by $\Pi_C$; thus, for $x\in\R^d$, $\Pi_C(x)$ is the unique point in $C$ for which $\|x-\Pi_C(x)\|\le\|x-y\|$ for all $y\in C$. The nearest-point map of a convex cone satisfies the homogeneity property $\Pi_C(\lambda x)=\lambda \Pi_C(x)$ for $x\in\R^d$ and $\lambda\ge 0$ and the Moreau decomposition
$$ \Pi_C(x) + \Pi_{C^\circ}(x) =x,\quad x\in\R^d,$$
with
$$ \langle \Pi_C(x),\Pi_{C^\circ}(x)\rangle =0.$$

Let $C\in\PC^d$ be a polyhedral cone. Then all faces of $C$ are polyhedral cones. We denote by $\F_j(C)$ the set of $j$-dimensional faces of $C$, $j\in\{0,\dots,d\}$ (possibly, $\F_j(C)=\emptyset$), and we write $\F(C):=\bigcup_{j=0}^d \F_j(C)$ for the set of all faces of $C$. 

For a face $F$ of $C$, we denote by $N(C,F)$ the normal cone of $C$ at $F$. The polar cone $C^\circ$ is again polyhedral. For $F\in\F_j(C)$ we have $N(C,F)\in\F_{d-j}(C^\circ)$ and $N(C^\circ,N(C,F))=F$.

A polyhedral cone $C$ induces a decomposition of $\R^d$, according to
\begin{equation}\label{3.1}
\sum_{F\in\F(C)} {\mathbbm 1}_{({\rm relint} F)+N(P,F)}(x)=1\quad\mbox{for all }x\in\R^d.
\end{equation}
Here ${\rm relint}$ denotes the relative interior. 

Let $C\in\PC^d$. To introduce the conic intrinsic volumes of $C$, we first define, for each face $F\in\F(C)$, the number
$$ v_F(C) = \gamma_d(F+N(C,F)).$$
A probabilistic interpretation is often convenient (as suggested, e.g., in \cite{MT14b}). For this, we denote by $\g$ a standard Gaussian random vector in $\R^d$, that is, a random variable (on some probability space) with values in $\R^d$ and distribution $\gamma_d$. Denoting (as before) probability by $\bP$, we then also have
$$ v_F(C)= \bP\{\Pi_C(\g)\in{\rm relint}\,F\}.$$
The {\em conic intrinsic volumes} $V_0(C),\dots,V_n(C)$ of $C$ are now defined by 
$$ V_k(C):= \sum_{F\in\F_k(C)} v_F(C)$$
for $k=0,\dots,d$, with $V_k(C):=0$ if $\F_k(C)=\emptyset$. (Since we consider only cones in this paper, there is no danger of confusion with the intrinsic volumes of convex bodies.) Defining
$$ {\rm skel}_k C:= \bigcup_{F\in\F_k(C)}{\rm relint}\,F$$
for $k=0,\dots,d$, we can simply write
$$ V_k(C) =\bP\{\Pi_C(\g)\in{\rm skel}_kC\}.$$
As a consequence of this, or of the decomposition (\ref{3.1}), the conic intrinsic volumes satisfy
\begin{equation}\label{3.1a}
\sum_{k=0}^d V_k=1.
\end{equation}

For better comparison with existing literature (most notably McMullen \cite{McM75}), we recall the definitions of the internal angle $\beta(o,F)$ of $F\in \F_k(C)$ at $o$ and the external angle $\gamma(F,C)$ of $C$ at $F$, namely
$$ \beta(o,F):= \gamma_{L(F)}(F),\qquad\gamma(F,C):= \gamma_{L(F)^\perp}(N(C,F)).$$
The property (\ref{2.0}) of the Gaussian measure gives
$$ v_F(C)=\gamma_d(F+N(C,F)) =\gamma_{L(F)}(F)\gamma_{L(F)^\perp}(N(C,F)) =\beta(o,F)\gamma(F,C),$$
hence
$$ V_k(C)=\sum_{F\in\F_k(C)}\beta(o,F)\gamma(F,C).$$

The conic intrinsic volumes have local versions, in the form of measures on the measurable space $(\R^d\times\R^d,\widehat\B(\R^d\times\R^d))$. For $C\in\PC^d$, the {\em support measures} $\Omega_0(C,\cdot),\dots,\Omega_d(C,\cdot)$ can be defined by
$$ \Omega_k(C,\eta):= \bP\{\Pi_C(\g)\in{\rm skel}_kC,\,(\Pi_C(\g),\Pi_{C^\circ}(\g))\in\eta\}$$
for $\eta\in\widehat\B(\R^d\times\R^d)$. Due to the decomposition (\ref{3.1}), we have the explicit representation
\begin{equation}\label{3.2} 
\Omega_k(C,\eta)= \sum_{F\in\F_k(C)} \int_F\int_{N(C,F)} {\mathbbm 1}_\eta(x,y)\,\gamma_{L(F)^\perp}(\D y)\,\gamma_{L(F)}(\D x).
\end{equation}
Clearly, $\Omega_k(C,\cdot)$ is a measure on the biconic $\sigma$-algebra $\widehat\B(\R^d\times\R^d)$, and
$$ \Omega_k(C,\R^d\times\R^d) =V_k(C)$$
for $k=0,\dots,d$.

Due to the invariance properties of the Gaussian measure, it follows from (\ref{3.2}) that the functions $\Omega_k$ are ${\rm O}_d$-equivariant, in the following sense. Defining
$$ \vartheta\eta:= \{(\vartheta x,\vartheta y): (x,y)\in\eta\}\quad\mbox{for } \eta\subset\R^d\times\R^d,\,\vartheta\in{\rm O}_d,$$
we have
$$
\Omega_k(\vartheta C,\vartheta \eta)= \Omega_k(C,\eta)
$$
for $C\in\PC^d$, $\eta\in\widehat\B(\R^d\times\R^d)$, and $\vartheta\in{\rm O}_d$.

By marginalization of the support measures, we obtain the {\em curvature measures}. They are defined by
$$ \Phi_k(C,A) := \Omega_k(C,A\times\R^d)$$
for $C\in\PC^d$, $A\in\widehat\B(\R^d)$, and $k=0,\dots,d$. Thus,
$$ \Phi_k(C,A) = \sum_{F\in\F_k(C)} \gamma_{L(F)}(F\cap A)\gamma(F,C).$$
The equivariance property reads
$$ \Phi_k(\vartheta C,\vartheta A) = \Phi_k(C,A)$$
for $C\in\PC^d$, $A\in\widehat\B(\R^d)$, and $\vartheta\in{\rm O}_d$. 
The functions $\Phi_k$ are intrinsic, in the following sense. If $A\in\widehat\B(L(C))$, then $\Phi_k(C,A)$ does not depend on the space containing $C$ in which they are computed. In particular, the conic intrinsic volumes of $C$ are independent of the dimension of the ambient space.

We explain the relation to the notions used in the spherical setting in \cite{Gla95} and \cite{SW08}. Let  a nonempty, closed, spherically convex set $K\subseteq\Sd$ and a Borel set $A\in\B(\Sd\times\Sd)$ be given. Define $ K^\vee:= \{\lambda x: x\in K,\,\lambda\ge 0\}$. Then the relation between the spherical support measures $\Theta_0,\dots,\Theta_{d-1}$ considered in \cite{Gla95}, \cite{SW08} and the conic support measures is given by 
$$\Theta_{k-1}(K,\eta)= \Omega_k(K^\vee,\eta^+)\quad\mbox{for }k=1,\dots,d.$$ 
The spherical intrinsic volumes $v_0,\dots, v_{d-1}$ appearing in \cite{Gla95}, \cite{SW08} are related to the conic intrinsic volumes by 
$$v_{k-1}(K)= V_k(K^\vee)\quad\mbox{for } k=0,\dots,d.$$

To deal with the curvature measures, we need a localization of the functional $v_F(C)$. For $C\in\PC^d$, a face $F$ of $C$ and a conic Borel set $A\in\widehat\B(\R^d)$, we define
$$ \varphi_F(C,A):= \bP\{\Pi_C(\g)\in A\cap{\rm relint}\,F\}.$$
Then, for $F\in\F_k(C)$,
\begin{eqnarray}
\varphi_F(C,A) &=& \gamma_d((A\cap F)+N(C,F)) = \gamma_{L(F)}(A\cap F)\gamma_{L(F)^\perp}(N(C,F))\nonumber\\
&=& \Phi_k(F,A)V_{d-k}(N(C,F)),\label{3.8}
\end{eqnarray}
and we have
\begin{equation}\label{3.9} 
\Phi_k(C,A)= \sum_{F\in\F_k(C)} \varphi_F(C,A).
\end{equation}

\vspace{2mm}

We conclude these preparations with some lemmas on the intersections of subspaces or cones, which will be needed in the next section. Let $L,L'\subseteq\R^d$ be subspaces. They are said to be {\em in general position} if
$$ \dim(L\cap L')= \max\{0,\dim L+\dim L'-d\}.$$
They are {\em in special position} if and only if they are not in general position, and this is equivalent to
$$\lin(L\cup L')\not=\R^d\quad\mbox{and}\quad L\cap L'\not=\{o\}.$$

\begin{lemma}\label{L3.1} 
Let $L,L'\subseteq \R^d$ be subspaces. The set of all rotations $\vartheta\in{\rm SO}_d$ for which $L$ and $\vartheta L'$ are in special position has $\nu$-measure zero.
\end{lemma}

Two different elementary proofs can be found in \cite[Lemma 4.4.1]{Sch14} and \cite[Lemma 13.2.1]{SW08}.

Let $C,D\in\PC^d$. The cones {\em intersect transversely}, written $C\pitchfork D$, if 
$$ \dim(C\cap D)=\dim C+\dim D-d \quad\mbox{and} \quad {\rm relint}\,C \cap {\rm relint}\,D \not=\emptyset.$$ 

The following is (part of) Proposition 2.5 in \cite{AL15}.

\begin{lemma}\label{L3.1a}
Let $C,D\in\PC^d$ be polyhedral cones. Every face of $C\cap D$ is of the form $F\cap G$ with $F\in\F(C)$ and $G\in\F(D)$. The normal cones satisfy
\begin{equation}\label{1.2.11}
N(C\cap D,F\cap G)\supseteq N(C,F)+N(D,G).
\end{equation}
If ${\relint}\,F\cap{\rm relint}\,G\not=\emptyset$, then equality holds in $\rm (\ref{1.2.11})$. If $F\pitchfork G$, then 
$$ N(C\cap D,F\cap G)=N(C,F)\oplus N(D,G)$$
is a direct sum.
\end{lemma}

For the following, we refer to \cite[Lemma 5.3]{AL15}. 

\begin{lemma}\label{L3.2} 
Let $C,D\in\PC^d$ be polyhedral cones. The set
$$ \{\vartheta\in{\rm SO}_d: C\cap \vartheta D=\{o\} \mbox{ or } C\pitchfork \vartheta D\}$$
has $\nu$-measure one. If $C$ and $D$ are not both subspaces, then the set
$$ \{\vartheta\in{\rm SO}_d: C\cap \vartheta D=\{o\} \mbox{ or } C\cap \vartheta D \mbox{ is not a subspace}\}$$
has $\nu$-measure one. 
\end{lemma}

The next lemma localizes assertion (2.8) in \cite{AL15}. It follows immediately from (\ref{3.8}) and Lemma \ref{L3.1a}.

\begin{lemma}\label{L3.3}
Let $C,D\in\PC^d$, $A,B\in\widehat\B(\R^d)$, $F\in \F_i(C)$, $G\in\F_j(D)$, and suppose that $i+j=d+k$ with $k>0$. If $F\pitchfork G$, then
$$ \varphi_{F\cap G}(C\cap D,A\cap B) =\Phi_k(F\cap G,A\cap B)V_{d-k}(N(C,F)+N(D,G)).$$
\end{lemma}

\section{Kinematic formulas}\label{sec4}
 
In this section, we prove Theorem \ref{T1.1}. We need not prove here that the integrand, the function $\vartheta\mapsto \Phi_k(C\cap\vartheta D, A\cap\vartheta B)$, is in fact $\nu$-integrable, since this is carried out, in spherical space, in \cite{Gla95} and \cite{SW08}, and can easily be carried over to the conic setting.

As auxiliary tools, we first prove two special cases. These will contain an additional term $f([C,\vartheta D])$. For this, we recall the {\em generalized sine function} of two subspaces $L_1,L_2\subseteq \R^d$. If $\dim L_1+\dim L_2=m\le d$, we choose an orthonormal basis in each $L_i$ and define $[L_1,L_2]$ as the $m$-dimensional volume of the parallelepiped spanned by the union of these bases. If one of the subspaces has dimension zero, then $[L_1,L_2]=1$, by definition. Clearly $[L_1,L_2]$ depends only on $L_1$ and $L_2$ and not on the choice of the bases. If $\dim L_1+\dim L_2 \ge d$, we define $[L_1,L_2]= [L_1^\perp,L_2^\perp]$ (which is consistent if $\dim L_1+\dim L_2=d$). For cones $C,D\in\C^d$, we define $[C,D]:= [L(C),L(D)]$. Let $f:[0,1]\to \R$ be a bounded, measurable function. If $\dim C=i$, $\dim D=j$, then the constant
\begin{equation}\label{4.1} 
c_{ij}(f):= \int_{{\rm SO}_d} f([C,\vartheta D])\,\nu(\D\vartheta), 
\end{equation}
depends only on $i,j,f$. This follows from the invariance properties of the measure $\nu$ and the invariance property $[\vartheta L_1,\vartheta L_2]=[L_1,L_2]$ for $\vartheta\in{\rm SO}_d$. 
If $f$ is a nonnegative power, then the constant $c_{ij}(f)$ has been computed in \cite[Lemma 4.4]{HSS08}. 

The reason for extending the kinematic formulas (\ref{4.2}) and (\ref{4.4}) by inserting $f([C,\vartheta D])$ lies in the fact that special cases of such extensions are useful in Euclidean integral geometry. For example, formula (\ref{4.4}) below can serve to prove \cite[Lemma 4.4.4]{Sch14} in a direct way. 

The method employed below, of averaging over suitable subgroups of the rotation group, was, in the context of polyhedral integral geometry, apparently first used by Amelunxen and Lotz \cite{AL15}. 

First we treat a special case of the kinematic formula, where the index of $\Phi_k(C\cap\vartheta D,\cdot)$ coincides with $\dim C+\dim D-d$.

\begin{theorem}\label{T4.1a}
Let $C,D\in\PC^d$ be polyhedral cones with $\dim C=i$, $\dim D=j$, where $i+j=d+k>d$. Let $A,B\in\widehat\B(\R^d)$, and let $f:[0,1]\to \R$ be a bounded, measurable function. Then
\begin{equation}\label{4.2}
\int_{{\rm SO}_d} \Phi_k(C\cap\vartheta D, A\cap\vartheta B)f([C,\vartheta D])\,\nu(\D\vartheta) = c_{ij}(f)\,\Phi_i(C,A)\,\Phi_j(D,B).
\end{equation}
\end{theorem}

\begin{proof}
It follows from Lemma \ref{L3.1} that for $\nu$-almost all $\vartheta\in{\rm SO}_d$ we have $\dim(L(C)\cap\vartheta L(D))=k$. Therefore,
\begin{eqnarray*}
&& \int_{{\rm SO}_d} \Phi_k(C\cap\vartheta D, A\cap\vartheta B)f([C,\vartheta D])\,\nu(\D\vartheta) \\
&& =  \int_{{\rm SO}_d} \gamma_{L(C)\cap\vartheta L(D)}(C\cap\vartheta D \cap A\cap\vartheta B)f([C,\vartheta D])\,\nu(\D\vartheta)\\
&&=  \int_{{\rm SO}_d} \int_{L(C)\cap\vartheta L(D)} {\mathbbm 1}_{C\cap A}(x){\mathbbm 1}_{\vartheta(D\cap B)}(x)f([C,\vartheta D])\,\gamma_{L(C)\cap \vartheta L(D)}(\D x)\,\nu(\D\vartheta).
\end{eqnarray*}
Here we replace $\vartheta$ by $\vartheta \rho^{-1}$ with $\rho\in {\rm SO}_{L(D)}$ (which satisfies $\rho L(D)=L(D)$ and hence $[C,\vartheta\rho^{-1} D] =[C,\vartheta D]$). This does not change the integral, by the invariance of $\nu$, hence the same holds if we integrate over all $\rho$ with respect to the probability measure $\nu_{L(D)}$. Thus, we obtain, using Fubini's theorem,
\begin{eqnarray*}
&& \int_{{\rm SO}_d} \Phi_k(C\cap\vartheta D, A\cap\vartheta B)f([C,\vartheta D])\,\nu(\D\vartheta) \\
&&= \int_{{\rm SO}_{L(D)}}  \int_{{\rm SO}_d} \int_{L(C)\cap\vartheta L(D)} {\mathbbm 1}_{C\cap A}(x){\mathbbm 1}_{D\cap B}(\rho\vartheta^{-1}x)f([C,\vartheta D])\\
&& \hspace{4mm}\times\,\gamma_{L(C)\cap \vartheta L(D)}(\D x)\,\nu(\D\vartheta)\,\nu_{L(D)}(\D\rho)\\
&&= \int_{{\rm SO}_d}  \int_{L(C)\cap\vartheta L(D)} {\mathbbm 1}_{C\cap A}(x)\left[ \int_{{\rm SO}_{L(D)}} {\mathbbm 1}_{D\cap B}(\rho\vartheta^{-1}x)  \,\nu_{L(D)}(\D\rho)\right] \\
&& \hspace{4mm} \times\,\gamma_{L(C)\cap \vartheta L(D)}(\D x)f([C,\vartheta D])\,\nu(\D\vartheta)\\
&&= \Phi_j(D,B) \int_{{\rm SO}_d}  \int_{L(C)\cap\vartheta L(D)} {\mathbbm 1}_{C\cap A}(x)\,\gamma_{L(C)\cap \vartheta L(D)}(\D x)f([C,\vartheta D])\,\nu(\D\vartheta).
\end{eqnarray*}
To evaluate the integral in brackets, we have applied (\ref{2.4}) in $L(D)$ and used that $ \Phi_j(D,B)=\gamma_{L(D)}(D\cap B)$, since $\dim D=j$. The application of (\ref{2.4}) is possible since $\vartheta^{-1}x\in L(D)$ and $x\not=0$ for $\gamma_{L(C)\cap \vartheta L(D)}$-almost all $x$. 

The latter outer integral does not change if we replace $\vartheta$ by $\sigma\vartheta$ with $\sigma\in{\rm SO}_{L(C)}$. Therefore, and since $L(C)=\sigma L(C)$ and $[C,\sigma\vartheta D]= [\sigma C,\sigma\vartheta D]=[C,\vartheta D]$, we get
\begin{eqnarray*}
&& \int_{{\rm SO}_d} \Phi_k(C\cap\vartheta D, A\cap\vartheta B)f([C,\vartheta D])\,\nu(\D\vartheta) \\
&& = \Phi_j(D,B)  \int_{{\rm SO}_{L(C)}} \int_{{\rm SO}_d} \int_{\sigma(L(C)\cap \vartheta L(D))} {\mathbbm 1}_{C\cap A}(x) \\
&& \hspace*{4mm}\times \gamma_{\sigma(L(C)\cap\vartheta L(D))} (\D x)f([C,\vartheta D])\, \nu(\D\vartheta)\,\nu_{L(C)}(\D\sigma)\\
&& = \Phi_j(D,B)  \int_{{\rm SO}_{L(C)}} \int_{{\rm SO}_d} \int_{L(C)\cap \vartheta L(D)} {\mathbbm 1}_{C\cap A}(\sigma x) \\
&& \hspace*{4mm}\times \gamma_{L(C)\cap\vartheta L(D)} (\D x)f([C,\vartheta D])\, \nu(\D\vartheta)\,\nu_{L(C)}(\D\sigma)\\
&& = \Phi_j(D,B)  \int_{{\rm SO}_d} \int_{L(C)\cap \vartheta L(D)} \left[\int_{{\rm SO}_{L(C)}}{\mathbbm 1}_{C\cap A}(\sigma x)\,\nu_{L(C)}(\D\sigma)\right]\\
&& \hspace*{4mm}\times  \gamma_{L(C)\cap\vartheta L(D)} (\D x)f([C,\vartheta D])\, \nu(\D\vartheta)\\
&& = c_{ij}(f)\Phi_j(D,B) \Phi_i(C,A),
\end{eqnarray*}
where we have used (\ref{2.4}) in $L(C)$. This completes the proof of (\ref{4.2}). 
\end{proof}

Together with the preceding theorem, the following one will be needed for the proof of the general kinematic formula. For $f\equiv 1$, the proof is due to Amelunxen and Lotz \cite{AL15}; we reformulate it here, for completeness and with some additional explanations.

\begin{theorem}\label{T4.b}
Let $C,D\in\PC^d$ be polyhedral cones with $\dim C=i$, $\dim D=j$, where $i+j=d-k<d$. Let $f:[0,1]\to \R$ be a bounded, measurable function. Then
\begin{equation}\label{4.4}
\int_{{\rm SO}_d} V_{d-k}(C+\vartheta D)f([C,\vartheta D])\,\nu(\D\vartheta) = c_{ij}(f)V_i(C)V_j(D).
\end{equation}
\end{theorem}

\begin{proof}
By Lemma \ref{L3.1}, we have $L(C)\cap\vartheta L(D) =\{o\}$ and $\dim(L(C)+\vartheta L(D)=i+j=d-k$ for $\nu$-almost all $\vartheta$. Using (\ref{2.0}), we get
\begin{eqnarray*}
V_{d-k}(C+\vartheta D) &=&\int_{L(C)+\vartheta L(D)} {\mathbbm 1}_{C+\vartheta D}(x)\,\gamma_{L(C)+\vartheta L(D)} (\D x)\\
&=& \int_{\R^d} {\mathbbm 1}_{C+\vartheta D +(L(C)+\vartheta L(D))^\perp}(x)\,\gamma_d(\D x).
\end{eqnarray*}
For $\nu$-almost all $\vartheta$, there is a unique decomposition
\begin{equation}\label{NN3} 
x = x_{C,\vartheta} + x_{D,\vartheta}+x_\vartheta
\end{equation}
with $x_{C,\vartheta}\in L(C)$, $x_{D,\vartheta}\in\vartheta L(D)$, $x_\vartheta\in (L(C)+\vartheta L(D))^\perp$, which we call the $\vartheta$-decomposition of $x$. By the uniqueness of the decomposition,
$$ x\in C+\vartheta D +(L(C)+ \vartheta L(D))^\perp \enspace \Leftrightarrow\enspace  x_{C,\vartheta}\in C \mbox{ and } x_{D,\vartheta}\in\vartheta D,$$
hence
\begin{eqnarray}
&& \int_{{\rm SO}_d} V_{d-k}(C+\vartheta D)f([C,\vartheta D])\,\nu(\D\vartheta) \nonumber\\
&& = \int_{{\rm SO}_d} \int_{\R^d} {\mathbbm 1}_{C+\vartheta D +(L(C)+ \vartheta L(D))^\perp}(x)\,\gamma_d(\D x)f([C,\vartheta D])\,\nu(\D\vartheta)\nonumber\\
&& = \int_{{\rm SO}_d} \int_{\R^d} {\mathbbm 1}_C(x_{C,\vartheta}){\mathbbm 1}_{\vartheta D}(x_{D,\vartheta})\,\gamma_d(\D x)f([C,\vartheta D])\,\nu(\D\vartheta).\label{NN1}
\end{eqnarray}

Now let $\rho\in {\rm SO}_{L(C)}$. Applying $\rho$ to both sides of (\ref{NN3}), we get
$$ \rho x = \rho x_{C,\vartheta} + \rho x_{D,\vartheta}+ \rho x_\vartheta $$
with $\rho x_{C,\vartheta}\in L(C)$, $\rho x_{D,\vartheta}\in \rho \vartheta L(D)$, $\rho x_\vartheta\in (L(C)+ \rho\vartheta L(D))^\perp$, because $\rho L(C)=L(C)$. On the other hand, the $\rho\vartheta$-decomposition of $\rho x$ reads
$$ \rho x = (\rho x)_{C,\rho\vartheta} + (\rho x)_{D,\rho\vartheta} + (\rho x)_{\rho\vartheta}$$
with $ (\rho x)_{C,\rho \vartheta} \in L(C)$, $(\rho x)_{D,\rho\vartheta} \in \rho\vartheta L(D)$, $(\rho x)_{\rho \vartheta}\in (L(C)+\rho\vartheta L(D))^\perp$. Thus, the two decompositions are identical.

The integral (\ref{NN1}) does not change if we replace $\vartheta$ by $\rho\vartheta$ and $x$ by $\rho x$, by the invariance properties of $\nu$ and $\gamma_d$. Further, we have $[C,\rho\vartheta D]=[\rho C,\rho \vartheta D] = [C,\vartheta D]$. Therefore,
\begin{eqnarray}
&& \int_{{\rm SO}_d} \int_{\R^d} {\mathbbm 1}_C(x_{C,\vartheta}){\mathbbm 1}_{\vartheta D}(x_{D,\vartheta})\,\gamma_d(\D x)f([C,\vartheta D])\,\nu(\D\vartheta)\nonumber\\
&& =\int_{{\rm SO}_d} \int_{\R^d} {\mathbbm 1}_C((\rho x)_{C,\rho\vartheta}){\mathbbm 1}_{\rho\vartheta D}((\rho x)_{D,\rho\vartheta})\,\gamma_d(\D x)f([C,\rho\vartheta D])\,\nu(\D\vartheta)\nonumber\\
&& =\int_{{\rm SO}_d} \int_{\R^d} {\mathbbm 1}_C(\rho x_{C,\vartheta}){\mathbbm 1}_{\rho\vartheta D}(\rho x_{D,\vartheta})\,\gamma_d(\D x)f([C,\vartheta D])\,\nu(\D\vartheta)\nonumber\\
&& = \int_{{\rm SO}_L(C)} \int_{{\rm SO}_d} \int_{\R^d} {\mathbbm 1}_C(\rho x_{C,\vartheta}){\mathbbm 1}_{\vartheta D}(x_{D,\vartheta})\,\gamma_d(\D x)f([C,\vartheta D])\,\nu(\D\vartheta)\,\nu_{L(C)}(\D \rho)\nonumber\\
&&= \int_{{\rm SO}_d} \int_{\R^d} \left[ \int_{{\rm SO}_{L(C)}} {\mathbbm 1}_C(\rho x_{C,\vartheta})\,\nu_{L(C)}(\D \rho))\right] {\mathbbm 1}_D(\vartheta^{-1} x_{D,\vartheta})\,\gamma_d(\D x)f([C,\vartheta D])\,\nu(\D \vartheta) \nonumber\\
&&= V_i(C) \int_{{\rm SO}_d} \int_{\R^d} {\mathbbm 1}_{D}(\vartheta ^{-1} x_{D,\vartheta})\,\gamma_d(\D x)f([C,\vartheta D])\,\nu(\D\vartheta).\label{NN4}
\end{eqnarray}
We  have applied (\ref{2.4}) in $L(C)$, which is possible since $x_{C,\vartheta}\in L(C)$ and $x_{C,\vartheta}\not= o$ for $\nu$-almost all $\vartheta$ and $\gamma_d$-almost all $x$.

Let $\sigma\in{\rm SO}_{L(D)}$. The integral in (\ref{NN4}) does not change if we replace $\vartheta$ by $\vartheta\sigma^{-1}$. Since $\sigma^{-1}L(D)=L(D)$, we have $ x_{D,\vartheta \sigma^{-1}}=x_{D,\vartheta}$. Further, $[C,\vartheta\sigma^{-1}D]=[C,\vartheta D]$. Therefore, we obtain
\begin{eqnarray}
&&  \int_{{\rm SO}_d} \int_{\R^d} {\mathbbm 1}_{D}(\vartheta^{-1} x_{D,\vartheta})\,\gamma_d(\D x) f([C,\vartheta D])\,\nu(\D\vartheta)\nonumber\\
&& = \int_{{\rm SO}_{L(D)}}  \int_{{\rm SO}_d} \int_{\R^d} {\mathbbm 1}_{D}(\sigma \vartheta ^{-1}x_{D,\vartheta\sigma^{-1}})\,\gamma_d(\D x)f([C,\vartheta \sigma^{-1} D])\,\nu(\D\vartheta)\,\nu_{L(D)}(\D \sigma)\nonumber\\
&& = \int_{{\rm SO}_d} \int_{\R^d} \left[ \int_{{\rm SO}_{L(D)}}  {\mathbbm 1}_{D}(\sigma \vartheta^{-1} x_{D,\vartheta})\,\nu_{L(D)}(\D \sigma)\right]\gamma_d(\D x)f([C,\vartheta D])\,\nu(\D\vartheta)\nonumber\\
&& = c_{ij}(f)V_j(D),\label{NN5}
\end{eqnarray}
where we have used (\ref{2.4}) in $L(D)$, which is possible since $\vartheta^{-1} x_{D,\vartheta}\in L(D)$. The results (\ref{NN1}), (\ref{NN4}), (\ref{NN5}) together complete the proof.
\end{proof}

\vspace{2mm}

\noindent{\em Proof of Theorem} \ref{T1.1}

\vspace{2mm}

By (\ref{3.9}),
$$ \Phi_k(C\cap\vartheta D,A\cap\vartheta B) =\sum_{J\in\F_k(C\cap\vartheta D)} \varphi_J(C\cap\vartheta D,A\cap\vartheta B).$$
By Lemmas \ref{L3.1a} and \ref{L3.2} (applied to all pairs of faces of $C$ and $D$), it holds for $\nu$-almost all $\vartheta\in{\rm SO}_d$ that each $k$-face $J$ of $C\cap \vartheta D$ is of the form $J=F\cap \vartheta G$ with $F\in\F_i(C)$, $G\in\F_j(D)$, $i+j=d+k$, and $F\pitchfork \vartheta G$, that is, $F$ and $\vartheta G$ intersect transversely. Therefore,
\begin{eqnarray*}
& & \int_{{\rm SO}_d} \Phi_k(C\cap \vartheta D,A\cap\vartheta B)\,\nu(\D\vartheta) \\
&& = \sum_{i+j=k+d} \sum_{F\in\F_i(C)} \sum_{G\in\F_j(D)} \int_{{\rm SO}_d} \varphi_{F\cap\vartheta G}(C\cap \vartheta D,A\cap\vartheta B){\mathbbm 1}\{F\pitchfork \vartheta G\}\,\nu(\D\vartheta).
\end{eqnarray*}
Hence, if we show that
$$
\int_{{\rm SO}_d} \varphi_{F\cap\vartheta G}(C\cap \vartheta D,A\cap\vartheta B){\mathbbm 1}\{F\pitchfork \vartheta  G\}\,\nu(\D\vartheta) 
= \varphi_F(C,A)\varphi_G(D,B),
$$
then the proof is complete. 

By Lemmas \ref{L3.2} and \ref{L3.3}, for $\nu$-almost all $\vartheta\in{\rm SO}_d$ we have
\begin{eqnarray*} 
&& \varphi_{F\cap\vartheta G}(C\cap\vartheta G,A\cap\vartheta B){\mathbbm 1}\{F\pitchfork \vartheta G\}\\
&&  =\Phi_k(F\cap \vartheta G,A\cap\vartheta B)V_{d-k}(N(C,F)+\vartheta N(D,G))
\end{eqnarray*}
(note that $V_{d-k}(N(C,F)+\vartheta N(D,G))=0$ if $\dim (F\cap \vartheta G)>k$).
Thus, we have to prove that
\begin{eqnarray}\label{4.7}
I &:= &\int_{{\rm SO}_d} \Phi_k(F\cap \vartheta G,A \cap\vartheta B)V_{d-k}(N(C,F)+\vartheta N(D,G))\,\nu(\D\vartheta)\nonumber\\
&\hspace{3pt}=& \varphi_F(C,A)\varphi_G(D,B).
\end{eqnarray}
For the proof, let $F\in\F_i(C)$, $G\in\F_j(D)$ with $i+j=d+k>d$ be given. In the following, we first replace $\vartheta$ by $\rho\vartheta$ with $\rho\in {\rm SO}_{L(F)}$ (noting that $N(C,F)=\rho N(C,F)$), which does not change the integral, then integrate over all $\rho$ with respect to $\nu_{L(F)}$, and use the ${\rm SO}_d$-invariance of the $V_i$, and Fubini's theorem. We obtain
\begin{eqnarray*}
I &=& \int_{{\rm SO}_{L(F)}}  \int_{{\rm SO}_d}\Phi_k(F\cap \rho\vartheta G,A\cap\rho\vartheta B)V_{d-k}(N(C,F)+\rho\vartheta N(D,G))\\
&& \times\;\nu(\D\vartheta)\,\nu_{L(F)}(\D\rho)\\
&=& \int_{{\rm SO}_{L(F)} } \int_{{\rm SO}_d} \Phi_k(F\cap \rho\vartheta G, A\cap\rho\vartheta B)V_{d-k}(N(C,F)+\vartheta N(D,G))\\
&& \times\;\nu(\D\vartheta)\,\nu_{L(F)}(\D\rho)\\
&=& \int_{{\rm SO}_d} \left[\int_{{\rm SO}_{L(F)}} \Phi_k(F\cap \rho\vartheta G,A\cap\rho\vartheta B) \,\nu_{L(F)}(\D\rho)\right]\\
&& \times\; V_{d-k}(N(C,F)+\vartheta N(D,G))\, \nu(\D\vartheta).
\end{eqnarray*}
Denoting the integral in brackets by $[\cdot]$, we have 
$$ [\cdot] = \int_{{\rm SO}_{L(F)}} \Phi_k(F\cap \rho(\vartheta G\cap L(F)),A\cap\rho(\vartheta B\cap L(F))) \,\nu_{L(F)}(\D\rho).$$
(Note that $\Phi_k(C,A)= \Phi_k(C,A\cap C)$ and $\rho L(F)=L(F)$.) If $\dim (\vartheta G\cap L(F)) =k$, we can apply (\ref{4.2}) (with $f=1$) in $L(F)$ and get
$$ [\cdot] =\Phi_i(F,A)\Phi_k(\vartheta G\cap L(F), \vartheta B).$$
If  $\dim (\vartheta G\cap L(F)) <k$, this equation also holds, since both sides are zero. Thus, we obtain
\begin{eqnarray*}
I &=& \Phi_i(F,A) \int_{{\rm SO}_d} \Phi_k(\vartheta G\cap L(F),\vartheta B) V_{d-k}(N(C,F)+\vartheta N(D,G))\,\nu(\D \vartheta)\\
&=& \Phi_i(F,A) \int_{{\rm SO}_d} \Phi_k(G\cap \vartheta^{-1} L(F),B) V_{d-k}(N(C,F)+\vartheta N(D,G))\,\nu(\D \vartheta),
\end{eqnarray*}
by the ${\rm SO}_d$-equivariance of $\Phi_k$. 

The latter integral can be treated in a similar way, replacing $\vartheta$ by $\vartheta\sigma$ with $\sigma\in{\rm SO}_{L(G)}$ (and thus satisfying $\sigma N(D,G)=N(D,G)$ and $\sigma L(G)=L(G)$), and integrating over all $\sigma$ with respect to $\nu_{L(G)}$. In this way, and again using the ${\rm SO}_d$-equivariance of $\Phi_k$, we obtain
\begin{eqnarray*}
I &=& \Phi_i(F,A) \int_{{\rm SO}_{L(G)}} \int_{{\rm SO}_d} \Phi_k(\sigma G\cap \vartheta^{-1}L(F),\sigma B)V_{d-k}(N(C,F)+\vartheta N(D,G))\\
&&\times\; \nu(\D\vartheta)\,\nu_{L(G)}(\D\sigma)\\
&=& \Phi_i(F,A) \int_{{\rm SO}_d} \left[ \int_{{\rm SO}_{L(G)}} \Phi_k(\sigma G\cap \vartheta^{-1}L(F),\sigma B)\,\nu_{L(G)}(\D\sigma) \right]\\
&&  \times\; V_{d-k}(N(C,F)+\vartheta N(D,G))\,\nu(\D\vartheta).
\end{eqnarray*}
The integral in brackets, denoted  by $[\cdot]$, can be written as
$$ [\cdot]= \int_{{\rm SO}_{L(G)}} \Phi_k(\vartheta^{-1}L(F)\cap L(G)\cap \sigma G,\R^d\cap\sigma B)\,\nu_{L(G)}(\D\sigma).$$
For $\nu$-almost all $\vartheta$ we have $\dim(\vartheta^{-1}L(F)\cap L(G))=k$ and hence can apply (\ref{4.2}) in $L(G)$, to obtain
$$ [\cdot] = \Phi_k(\vartheta^{-1}L(F)\cap L(G),\R^d)\Phi_j(G,B) = \Phi_j(G,B).$$
Thus, we arrive at
\begin{eqnarray*}
I &=& \Phi_i(F,A)\Phi_j(G,B) \int_{{\rm SO}_d}V_{d-k}(N(C,F)+\vartheta N(D,G))\,\nu(\D\vartheta)\\
&=&\Phi_i(F,A)\Phi_j(G,B) V_{d-i}(N(C,F))V_{d-j}(N(D,G))\\
& =& \varphi_F(C,A)\varphi_G(D,B),
\end{eqnarray*}
where we have applied (\ref{4.4}) (with $f=1$). We have shown (\ref{4.7}), which finishes the proof.

\section{Random Schl\"afli cones}\label{sec5}

The notion of random cones requires a (Borel) $\sigma$-algebra on the set $\C^d$ of closed convex cones. For this, we introduce a topology. For $C,D\in\C^d$, let $\Delta(C,D) := \delta(C\cap B^d,D\cap B^d)$, where $B^d$ is the unit ball of $\Rd$ with center at the origin $o$, and where $\delta$ denotes the ordinary Hausdorff metric on convex bodies of $\R^d$. The induced topology coincides with the trace topology of the Fell topology on closed sets, as can be seen by combining \cite[Thm. 12.2.2]{SW08} and \cite[Thm. 1.8.8]{Sch14}. The Borel $\sigma$-algebra $\B(\C^d)$ is now well-defined.

We recall a class of random polyhedral cones, which have been introduced by Cover and Efron \cite{CE67} and studied more thoroughly in \cite{HS16}. Let $G(d,d-1)$ denote the Grassmannian of $(d-1)$-dimensional linear subspaces of $\R^d$. Let $H_1,\dots,H_n\in G(d,d-1)$ be in general position, that is, any $k\le d$ of the subspaces have an intersection of dimension $d-k$. The {\em Schl\"afli cones} induced by $H_1,\dots,H_n$ are the closures of the components of $\R^d\setminus\bigcup_{i=1}^n H_i$. Generally, we denote by $\F_d(H_1,\dots,H_n)$ the set of $d$-dimensional cones of the tessellation of $\R^d$ induced by the hyperplanes $H_1,\dots,H_n$.

We denote by $\sigma:= \sigma_{d-1}/\omega_d$ the normalized spherical Lebesgue measure and by $\sigma^*$ its image measure under the mapping $u\mapsto u^\perp$ from $\Sd$ to $G(d,d-1)$. Let $n\in\N$, and let $\Ha_1,\dots,\Ha_n$ be independent random hyperplanes with distribution $\sigma^*$. With probability one, these hyperplanes are in general position. The {\em isotropic random Schl\"afli cone} with parameter $n$, denoted by $S_n$, is obtained by selecting at random, with equal chances, one of the Schl\"afli cones induced by $H_1,\dots,H_n$. As made precise in \cite{HS16}, $S_n$ is a random cone with distribution given by
\begin{equation}\label{5.1}
\bP\{S_n\in B\} = \int_{(G(d,d-1))^n} \frac{1}{C(n,d)} \sum_{C\in\F_d(H_1,\dots,H_n)} {\mathbbm 1}_B(C)\,\sigma^{*n}(\D(H_1,\dots,H_n))
\end{equation}
for $B\in\B(\C^d)$ (with $C(n,d)$ given by (\ref{1.0})). In analogy to common terminology for stationary tessellations, the random Schl\"afli cone could also be called the {\em typical cone} of the tessellation induced by $\Ha_1,\dots,\Ha_n$.

\newpage

\noindent{\em Proof of Theorem} \ref{T1.2}

Let $C\in\PC^d$ be a polyhedral cone and let $A\in\widehat\B(\R^d)$. Let $n\in \N$ and $k\in\{1,\dots,d\}$. Theorem \ref{T1.2} states that
$$ \bE\Phi_k(C\cap S_n,A) = \frac{1}{C(n,d)} \sum_{s=0}^{\min\{n,d-k\}} \binom{n}{s} \Phi_{k+s}(C,A).$$

First we note that the measurability of the function $\Phi_k(C\cap S_n,A)$ follows from the facts that $S_n$, as defined in \cite{HS16}, is a random cone, the intersection mapping is upper semicontinuous (\cite[Thm. 12.2.6]{SW08}), and $\Phi_k(\cdot,A)$ is measurable, as can be deduced from \cite[Thm. 6.5.2]{SW08}. 

In order to use (\ref{5.1}), we write a linear hyperplane $H$ in the form $H=u^\perp$ with $u\in\Sd$. Then one of the two closed halfspaces bounded by $H$ is given by
$$ u^- := \{x\in\Rd:\langle x,u\rangle \le 0\}.$$
Let $u_1,\dots,u_n\in\Sd$. If the hyperplanes $u_1^\perp,\dots,u_n^\perp$ are in general position, then the Schl\"afli cones induced by these hyperplanes are precisely the cones different from $\{o\}$ of the form
$$ \bigcap_{i=1}^n \epsilon_iu_i^-,\quad \epsilon_i=\pm 1.$$

Let $f$ by a nonnegative measurable function $f$ on polyhedral cones satisfying $f(\{o\})=0$. By (\ref{5.1}), we have
\begin{eqnarray*} 
\bE f(S_n) &=& \int_{(G(d,d-1))^n} \frac{1}{C(n,d)} \sum_{C\in\F_d(H_1,\dots,H_n)} f(C)\,\sigma^{*n}(\D(H_1,\dots,H_n))\\
&=& \frac{1}{C(n,d)} \int_{(\Sd)^n} \sum_{\epsilon_1,\dots,\epsilon_n=\pm 1} f((\epsilon_1u_1)^-\cap\dots\cap (\epsilon_nu_n)^-)\,\sigma^n(\D(u_1,\dots,u_n))\\
&=& \frac{2^n}{C(n,d)} \int_{(\Sd)^n}  f(u_1^-\cap\dots\cap u_n^-)\,\sigma^n(\D(u_1,\dots,u_n)),
\end{eqnarray*}
since the hyperplanes $u_1^\perp,\dots,u_n^\perp$ are in general position for $\sigma^n$-almost all $(u_1,\dots,u_n)$ and the measure $\sigma$ is invariant under reflection in the origin.

Applying the preceding to the function $f= \Phi_k(C\cap \cdot\,,A)$, we obtain
$$ \bE\Phi_k(C\cap S_n,A)= \frac{2^n}{C(n,d)} \int_{(\Sd)^n}  \Phi_k(C\cap u_1^-\cap\dots\cap u_n^-,A)\,\sigma^n(\D(u_1,\dots,u_n)). $$

We set
$$ I(j,m):= \int_{(\Sd)^j} \Phi_m(C\cap\cap u_1^-\cap\dots\cap u_j^-,A)\,\sigma^j(\D(u_1,\dots,u_j))$$
for $j=0,\dots,n$ and $m\ge k$, with $I(0,m):= \Phi_m(C,A)$ and $\Phi_m(C,A):=0$ for $m>\dim C$. Then we can write
$$ I(n,k)= \int_{(\Sd)^{n-1}} \int_{\Sd} \Phi_k(C'\cap u^-,A)\,\sigma(\D u)\,\sigma^{n-1}(\D(u_1,\dots,u_{n-1}))$$
with $C':= C\cap u_1^-\cap\dots\cap u_{n-1}^-$.
With a fixed closed halfspace $H^-$ (bounded by some $H\in G(d,d-1)$), we have
\begin{eqnarray*}
\int_{\Sd} \Phi_k(C'\cap u^-,A)\,\sigma(\D u) &=&
\int_{{\rm SO}_d} \Phi_k(C'\cap\vartheta H^-,A)\,\nu(\D\vartheta)\\
&= &\sum_{i=k}^d \Phi_i(C',A)V_{d+k-i}(H^-)\\
&=& \Phi_k(C',A)\cdot\frac{1}{2} + \Phi_{k+1}(C',A)\cdot\frac{1}{2},
\end{eqnarray*}
where we have used the kinematic formula of Theorem \ref{T1.1} and the fact that
$$ V_j(H^-) = \left\{ \begin{array}{ll} 0 & \mbox{if } j\le d-2,\\ 1/2 & \mbox{if } j=d-1,d,\end{array}\right.$$
as follows from the definition of the conic intrinsic volumes. Therefore,
$$ I(n,k) = \frac{1}{2}[I(n-1,k)+I(n-1,k+1)].$$
Now an induction argument gives
$$ I(n,k) = \frac{1}{2^n} \sum_{s =0}^n \binom{n}{s} I(0,k+s) = \frac{1}{2^n} \sum_{s = 0}^{\min\{n,d-k\}} \binom{n}{s} \Phi_{k+s}(C,A).$$
This yields the assertion. \hfill$\Box$

\vspace{2mm}

The proof of Theorem \ref{T1.3} rests on the facts that $S_n$ is isotropic, that is $S_n$ and $\vartheta S_n$ have the same distribution, for each $\vartheta\in{\rm SO}_d$, and that we have an explicit formula for the expectations of its conic intrinsic volumes. Let $\bQ_{S_n}$ be the distribution of $S_n$. Since 
$$ \bP\{C\cap S_n\not=\{o\}\} =\int_{\C^d} {\mathbbm 1}\{C\cap\vartheta D\not=\{o\}\}\,\bQ_{S_n}(\D D)$$
for each $\vartheta\in{\rm SO}_d$, we get
\begin{eqnarray*}
\bP\{C\cap S_n\not=\{o\}\} &=& \int_{{\rm SO}_d} \int_{\C^d} {\mathbbm 1}\{C\cap\vartheta D\not=\{o\}\}\,\bQ_{S_n}(\D D)\,\nu(\D\vartheta)\\
&=&  \int_{\C^d}\int_{{\rm SO}_d} {\mathbbm 1}\{C\cap\vartheta D\not=\{o\}\}\,\nu(\D\vartheta)\,\bQ_{S_n}(\D D)\\
&=&  \int_{\C^d} 2\sum_{k= 0}^{\lfloor\frac{d-1}{2}\rfloor} \sum_{i=2k+1}^d V_{d+2k+1-i}(C)V_i(D) \,\bQ_{S_n}(\D D)\\
&=&  2\sum_{k= 0}^{\lfloor\frac{d-1}{2}\rfloor} \sum_{i=2k+1}^d V_{d+2k+1-i}(C)\, \bE V_i(S_n),
\end{eqnarray*}
where (\ref{1.4}) was used. The expected conic intrinsic volumes of random Schl\"afli cones have been determined in \cite{HS16}. For $i=1,\dots,d$, we have (loc. cit., Corollary 4.3)
$$ \bE V_i(S_n) =\binom{n}{d-i} C(n,d)^{-1}.$$
This gives
$$
\bP\{C\cap S_n\not=\{o\}\} =  \frac{2}{C(n,d)} \sum_{j\ge 1} a_{n,j} V_j(C)
$$
with
$$ a_{n,j} = \sum_{k= 0}^{\lfloor\frac{j-1}{2}\rfloor} \binom{n}{j-2k-1},$$
which completes the proof of Theorem \ref{T1.3}

\vspace{2mm}

Now let $S_n$ and $T_m$ be two stochastically independent random Schl\"afli cones, with parameters $n$ and $m$ and distributions $\bQ_{S_n},\bQ_{T_m}$, respectively. By the independence of $S_n$ and $T_m$ we have
\begin{eqnarray*}
\bP \{S_n\cap T_m\not=\{o\}\} &=& \int_{C^d} \int_{\C^d} {\mathbbm 1}\{C\cap D\not=\{o\}\}\,\bQ_{S_n}(\D C)\,\bQ_{T_m}(\D D)\\
&=& \int_{C^d} \bE{\mathbbm 1}\{S_n\cap D\not=\{o\}\}\,\bQ_{T_m}(\D D)\\
&=& \int_{\C^d} \frac{2}{C(n,d)} \sum_{j\ge 1} a_{n,j} V_j(D)\,\bQ_{T_m}(\D D)\\
&=& \frac{2}{C(n,d)} \sum_{j\ge 1} a_{n,j} \,\bE V_j(T_m)\\
&=& \frac{2}{C(n,d)C(m,d)} \sum_{j\ge 1} \binom{m}{d-j}\sum_{k= 0}^{\lfloor\frac{j-1}{2}\rfloor} \binom{n}{j-2k-1}\\
&=& \frac{2}{C(n,d)C(m,d)} \sum_{k=0}^{\lfloor\frac{d-1}{2}\rfloor} \sum_{p+q=d-2k-1} \binom{n}{p}\binom{m}{q}.
\end{eqnarray*}
This is the assertion of Theorem \ref{T1.4}.

\vspace{3mm}

\noindent Author's address:\\[2mm]
Rolf Schneider\\
Mathematisches Institut, Albert-Ludwigs-Universit{\"a}t\\
D-79104 Freiburg i. Br., Germany\\
E-mail: rolf.schneider@math.uni-freiburg.de


\begin{thebibliography}{10}

\bibitem{Ame15} D. Amelunxen, Measures on polyhedral cones: characterizations and kinematic formulas. (Preprint) arXiv:1412.1569v2 (2015).

\bibitem{AB12} D. Amelunxen, P. B\"urgisser, Intrinsic volumes of symmetric cones. (Extended version of \cite{AB15}) arXiv:1205.1863 (2012).

\bibitem{AB15} D. Amelunxen, P. B\"urgisser, Intrinsic volumes of symmetric cones and applications in convex programming. {\em Math. Programm., Ser. A} {\bf 149}, 105--130  (2015).

\bibitem{AL15} D. Amelunxen, M. Lotz, Intrinsic volumes of polyhedral cones: a combinatorial perspective. (Preprint) arXiv:1512.06033 (2015).

\bibitem{ALMT14} D. Amelunxen, M. Lotz, M.B.  McCoy, J.A. Tropp, Living on the edge: phase transitions in convex programs with random data. {\em Inf. Inference} {\bf 3}, 224--294 (2014).

\bibitem{CE67} T.M. Cover, B. Efron, Geometrical probability and random points on a hypersphere. {\em Ann. of Math. Statist.} {\bf 38}, 213--220 (1967).

\bibitem{Gla95} S. Glasauer, Integralgeometrie konvexer K\"orper im sph\"arischen Raum. Doctoral Thesis, Albert-Ludwigs-Universit\"at, Freiburg i.~Br. (1995).\\Available from: www.hs-augsburg.de/$\sim$glasauer/publ/diss.pdf

\bibitem{Gla96} S. Glasauer, Integral geometry of spherically convex bodies. {\em Diss. Summ. Math.} {\bf 1}, 219--226 (1996).

\bibitem{HS16} D. Hug, R. Schneider, Random conical tessellations. {\em Discrete Comput. Geom.} {\bf 56}, 395--426  (2016).

\bibitem{HSS08} D. Hug, R. Schneider, R. Schuster, Integral geometry of tensor valuations. {\em Adv. Appl. Math.} {\bf 41}, 482--509  (2008).

\bibitem{HW16} D. Hug, J.A. Weis,  Kinematic formulae for tensorial curvature measures. (Preprint) arXiv:1612.08427 (2016).

\bibitem{MT14a} M.B. McCoy, J.A. Tropp, Sharp recovery bounds for convex demixing, with applications. {\em Found. Comput. Math.} {\bf 14}, 503--567  (2014).

\bibitem{MT14b} M.B. McCoy, J.A. Tropp, From Steiner formulas for cones to concentration of intrinsic volumes.  {\em Discrete Comput. Geom.} {\bf 51}, 926--963  (2014).

\bibitem{MT13} M.B. McCoy, J.A. Tropp, The achievable performance of convex demixing. (Preprint) arXiv:1309.7478 (2013).

\bibitem{McM75} P. McMullen, Non-linear angle-sum relations for polyhedral cones and polytopes. {\em Math. Proc. Camb. Philos. Soc.} {\bf 78}, 247--261  (1975).


\bibitem{San62a} L.A. Santal\'{o}, Sobre la formula de Gauss--Bonnet para poliedros en espacios de curvatura constante. {\em Revista Un. Mat. Argentina} {\bf 20}, 79--91  (1962).

\bibitem{San62b} L.A. Santal\'{o}, Sobre la formula fundamental cinematica de la geometria integral en espacios de curvatura constante. {\em Math. Notae} {\bf 18}, 79--94  (1962).

\bibitem{San76} L.A. Santal\'{o}, Integral Geometry and Geometric Probability. Encyclopedia of Mathematics and Its Applications, vol. {\bf 1\/}, Addison--Wesley, Reading, MA (1976).

\bibitem{Sch78} R. Schneider, Curvature measures of convex bodies. {\em Ann. Mat. Pura Appl.} {\bf 116}, 101--134  (1978).

\bibitem{Sch14} R. Schneider, {\em Convex Bodies: The Brunn-Minkowski Theory.} 2nd edn., Encyclopedia of Mathematics and Its Applications, vol. {\bf 151}, Cambridge University Press, Cambridge (2014).

\bibitem{SW08} R. Schneider, W. Weil, {\em Stochastic and Integral Geometry.} Springer, Berlin (2008).

\end{thebibliography}
\end{document}